\documentclass[12pt,twoside]{amsart}
\usepackage{amssymb}
\usepackage{amscd}

\title[The numerical class of 
a surface on a toric manifold]
{The numerical class of 
a surface on a toric manifold} 
\author{Hiroshi Sato} 
\date{}
\address{Faculty of Economics and Information\\ 
Gifu Shotoku Gakuen University, 
1-38 Nakauzura Gifu 500-8288 Japan}
\email{hirosato@gifu.shotoku.ac.jp}

\newcommand{\xSupp}[0]{{\operatorname{Supp}}}

\newcommand{\Pic}[0]{{\operatorname{Pic}}}

\newcommand{\G}{\mathop{\rm G}\nolimits}

\newcommand{\R}{\mathop{\mathbb{R}}\nolimits}
\newcommand{\bP}{\mathop{\mathbb{P}}\nolimits}
\newcommand{\N}{\mathop{\rm N}\nolimits}
\newcommand{\NE}{\mathop{\rm NE}\nolimits}
\newcommand{\Z}{\mathop{\rm Z}\nolimits}

\newcommand{\bZ}{\mathop{\mathbb{Z}}\nolimits}

\newtheorem{thm}{Theorem}[section]
\newtheorem{lem}[thm]{Lemma}

\newtheorem{prop}[thm]{Proposition}
\newtheorem{lem-def}[thm]{Definition-Lemma}

\theoremstyle{definition}
\newtheorem{ex}[thm]{Example}
\newtheorem{defn}[thm]{Definition}
\newtheorem{rem}[thm]{Remark}
\newtheorem*{ack}{Acknowledgments}       
 
\newtheorem*{notation}{Notation}         

\newtheorem{say}[thm]{}
\begin{document}
\bibliographystyle{amsalpha+}

\dedicatory{Dedicated to Professor~Shihoko Ishii on her~sixtieth~birthday}

\begin{abstract}
In this paper, we give a method to describe the numerical 
class of a torus invariant surface on a projective toric 
manifold. As applications, we can classify toric 2-Fano 
manifolds of Picard number 2 or of dimension at most 4.  
\end{abstract}

\maketitle
\tableofcontents

\section{Introduction}\label{intro}
\thispagestyle{empty}
The classification of smooth toric Fano $d$-folds is an important and 
interesting problem. They are classified for $d=3$ by \cite{bat0} and \cite{watanabe1}, 
for $d=4$ by \cite{bat4} and \cite{sato1}, and for $d=5$ by \cite{kreuzernill}. 
In \O bro's recent excellent paper \cite{obro}, an algorithm which classify all the 
smooth toric Fano $d$-folds for any given natural number $d$ was constructed. 
So, we can say that the classification of smooth toric Fano varieties is completed. 

On the other hand, de Jong and Starr defined a special class of Fano manifolds 
called {\em $2$-Fano manifolds} in \cite{starr} (see Definition \ref{2fanodef}). So, 
we consider the problem of the classification of toric $2$-Fano manifolds 
as a next step. For this classification, we give a method to describe 
the numerical class of a $2$-cycle on projective toric manifolds 
(see Section \ref{kushi}). 
This method make calculations of intersection numbers 
much easier. 
As results, we obtain the classification of 
toric $2$-Fano manifolds for the case of Picard number $\rho(X)=2$ and 
for the case of $\dim(X)\le 4$. We remark that Nobili classified 
smooth toric $2$-Fano $4$-folds in \cite{nobili} by using a Maple program. 

The contents of this paper is as follows: 
In Section \ref{junbi}, we define the basic notation such as 
{\em nef} $2$-cocycle and {\em $2$-Mori cone} 
for our theory.  In Section \ref{kushi}, we define a polynomial $I_{Y/X}$ for a 
torus invariant subvariety $Y\subset X$. This polynomial has all the informations of 
intersection numbers of $Y$ on $X$. 
So, we can consider this polynomial as 
the numerical class of $Y$. For a some special surface $S$, $I_{S/X}$ 
has a good property to calculate intersection numbers (see Theorems \ref{shuteiri} 
and \ref{shuteiri2}). As applications, we classify toric $2$-Fano manifolds 
under some assumptions in Section \ref{zelda}.

\begin{ack} 
The author would like to thank Professor Osamu Fujino 
for advice and encouragement. 
He was partially supported 
by the Grant-in-Aid for Scientific Research (C) $\sharp$23540062 from JSPS.
\end{ack}

\begin{notation}
We will work over an algebraically closed field $k$ 
throughout this paper. 
We denote a projective toric $d$-fold 
by $X=X_\Sigma$, where 
$\Sigma$ is the associated fan in $N:=\bZ^d$. 
$\G(\Sigma)\subset N$ is the set of the primitive generators 
for the $1$-dimensional cones in $\Sigma$. 
\end{notation}

\section{Preliminaries}\label{junbi}
In this section, we explain the notation and some basic 
facts of the toric geometry and the birational geometry used in this paper. See 
\cite{fujisato}, \cite{fulton} and \cite{oda} 
for the detail. 

Let $X$ be a smooth projective toric $d$-fold. 
Put $\Z_2(X)$ be the free $\mathbb{Z}$-module 
of $2$-cocycles on $X$ and 
$\Z^2(X)$ the free $\mathbb{Z}$-module of $2$-cycles on $X$. 
We define the {\em numerical equivalence} ``$\equiv$'' on 
$\Z^2(X)$ and $\Z_2(X)$: 
A $2$-cocycle $E\in\Z^2(X)$ is {\em numerically equivalent to $0$}, that is, 
$E\equiv 0$  
if the intersection number $(E\cdot S)=0$ for any $2$-cycle 
$S\in\Z_2(X)$, while a $2$-cycle $S\in\Z_2(X)$ is {\em numerically equivalent} to $0$, 
that is, $S\equiv 0$  
if the intersection number $(E\cdot S)=0$ for any $2$-cocycle 
$E\in\Z^2(X)$. 
We define $\N^2(X):=(\Z^2(X)/\equiv)\otimes \R$ and 
$\N_2(X):=(\Z_2(X)/\equiv)\otimes \R$.

The following definitions are similar to the case of 
divisors and curves: 

\begin{defn}\label{nefness}
A $2$-cocycle $E\in\Z^2(X)$ is a {\em nef} $2$-cocycle 
if $(E\cdot S)\ge 0$ for 
any $2$-cycle $S\in\Z_2(X)$.
\end{defn}

\begin{defn}\label{2moricone}
For a projective toric manifold $X$, let $\NE_2(X)\subset \N_2(X)$ be 
the cone of numerical effective $2$-cycles. Namely,
$$\NE_2(X):=\left\{\left.\left[\sum_i a_iS_i\right]\in\N_2(X)\,\right|\,
a_i\ge 0   \right\}.$$
We call $\NE_2(X)\subset \N_2(X)$ 
the {\em $2$-Mori cone} of $X$.
\end{defn}

We should remark that $\N^l(X)$, $\N_l(X)$ and $\NE_l(X)$ can be defined 
for any $1\le l\le d$ similarly. 
 
The following is an immediate consequence of 
the {\em projectivity} of $X$:
\begin{prop}
${\rm NE}_2(X)$ is a strongly convex cone.
\end{prop}
\begin{proof}
Let $D$ be an ample divisor on $X$. Then, for any 
$S\in{\rm NE}_2(X)\setminus\{0\}$, we have 
$(D^{2}\cdot S)>0$. Namely, ${\rm NE}_2(X)$ is 
strongly convex.
\end{proof}
On the other hand, 
for the toric case, the following is obvious:
\begin{prop}
Let $X$ be a smooth projective toric $d$-fold. Then, 
${\rm NE}_2(X)$ is a polyhedral cone.
\end{prop}

Thus, ${\rm NE}_2(X)$ is a strongly convex polyhedral 
rational cone similarly as ${\rm NE}(X)$. 

We end this section by giving the following simple examples:
\begin{ex}
\begin{enumerate}
\item If $X=\mathbb{P}^d$, then 
$$\NE_2(X)=\mathbb{R}_{\ge 0}[S],$$
where $S$ is a plane in $X$.
\item If $X=\mathbb{P}^1\times\mathbb{P}^3$, then 
$$\NE_2(X) =\mathbb{R}_{\ge 0}[(\mbox{a point})\times 
\mathbb{P}^2]+\mathbb{R}_{\ge 0}[\mathbb{P}^1\times 
\mathbb{P}^1].$$
\item If $X=\mathbb{P}^2\times\mathbb{P}^2$, then 
$$\NE_2(X) =\mathbb{R}_{\ge 0}[(\mbox{a point})\times 
\mathbb{P}^2]+\mathbb{R}_{\ge 0}[\mathbb{P}^1\times 
\mathbb{P}^1]+\mathbb{R}_{\ge 0}[\mathbb{P}^2\times 
(\mbox{a point})].$$
\end{enumerate}
\end{ex}

\section{Combinatorial descriptions}\label{kushi}
In this section, we establish a method to describe the numerical 
class of a torus invariant subvariety. 

Let $Y=Y_\sigma\subset X$ be a torus invariant subvariety 
of $\dim Y=l$  
associated to a cone $\sigma\in\Sigma$ 
and $\G(\Sigma)=\{x_1,\ldots,x_m\}$. 
Put 
$$I_{Y/X}=I_{Y/X}(X_1,\ldots,X_m):=
\sum_{1\le i_1,\ldots,i_l\le m}
(D_{x_{i_1}}\cdots D_{x_{i_l}}\cdot Y)
X_{i_1}\cdots X_{i_l}$$
$$\in\mathbb{Z}[X_1,\ldots,X_m],$$
where $D_{x_i}$ is the torus invariant prime divisor 
corresponding to $x_i$, while $X_i$ is defined to be 
the independent variable corresponding to $x_i$. 
We will use this notation throughout this paper. 

\begin{rem}
$I_{Y/X}$ has all the informations of intersection numbers 
of $Y$ on $X$. So, we can consider $I_{Y/X}$ as 
the numerical class of $Y\in \N_l(X)$.
\end{rem}

\begin{ex}
Let $C=C_{\tau}\subset X$ be a torus invariant curve, 
where $\tau$ is a $(d-1)$-dimensional cone, 
that is, a {\em wall} in $\Sigma$. 
In this case, 
$$I_{C/X}=\sum_i (D_i\cdot C)X_i$$ 
is a polynomial of degree $1$. 
On the other hand,
$$\sum_i (D_i\cdot C)x_i=0$$ 
is the so-called {\em Reid's wall relation }
associated to the wall $\tau$ (see \cite{reid}). 
Namely, $I_{C/X}$ is calculated from the wall relation 
immediately.
\end{ex}

\begin{ex}\label{hanahir}
When $Y=X$, $I_{X/X}$ sometimes becomes a simple shape as follows:  
\begin{enumerate}
\item {\em Projective spaces}: Let $X$ be the $d$-dimensional projective space  $\bP^d$ 
and $\G(\Sigma)=\{x_1:=e_1,\ldots,x_d:=e_d,x_{d+1}:=-(e_1+\cdots+e_d)\}$. 
Then, 
$$I_{X/X}=(X_1+\cdots+X_{d+1})^d.$$
\item {\em Hirzebruch surfaces}: Let $X$ be the Hirzebruch surface $F_\alpha$ 
of degree $\alpha$ and $\G(\Sigma)=
\{x_1:=e_1,x_2:=e_2,x_3:=-e_1+\alpha e_2,x_4=-e_2\}$. Then, 
$$I_{X/X}=\alpha(X_2+X_4)^2+2(X_2+X_4)(X_1+X_3-\alpha X_2).$$
\end{enumerate}
\end{ex}

Let $X$ be a smooth projective toric variety and 
$S\subset X$ be a torus invariant {\em surface}. 
For some special cases, $I_{S/X}$ is simply 
calculated as follows. These are the main theorems 
of this paper.

\begin{thm}\label{shuteiri}
Suppose $S\cong\mathbb{P}^2$. 
Let $C\subset S$ be a torus invariant curve. Then, 
$I_{S/X}=(I_{C/X})^2.$
\end{thm}

\begin{proof}
Let $\tau=\R_{\ge 0}x_1+\cdots+\R_{\ge 0}x_{d-2}\in\Sigma$ be 
the $(d-2)$-dimensional cone 
associated to $S=S_\tau$, where 
$\tau\cap\G(\Sigma)=\{x_1,\ldots,x_{d-2}\}$. 
Then, there exist exactly three 
maximal cones $\tau+\R_{\ge 0}y_1,\ \tau+\R_{\ge 0}y_2,\  
\tau+\R_{\ge 0}y_3\in\Sigma$ which contain $\tau$. 
Put $$y_1+y_2+y_3+a_1x_1+\cdots+a_{d-2}x_{d-2}=0$$ be 
the wall relation corresponding to $C$. 
For the proof, it is suffice to show that 
$$D_zD_wS=a_za_w$$ 
for any $z,\ w\in\G(\Sigma)$, 
where $D_z$ is the prime torus invariant divisor 
corresponding to $z$, while $a_z$ is the 
coefficient of $z$ in the above wall relation. 

Suppose that $z$ or $w\not\in\{x_1,\ldots,x_{d-2},y_1,y_2,y_3\}$. 
Namely, $a_z=0$ or $a_w=0$. In this case, trivially,  
$D_zS=0$ or $D_wS=0$. So,  $D_zD_wS=a_za_w=0$. 

For any $1\le i,j\le 3$, 
$$D_{y_i}D_{y_j}S=(D_{y_i}|_S)(D_{y_j}|_S)=C^2=1.$$
So, the remaining case is $z$ or $w\in\{x_1,\ldots,x_{d-2}\}$.  
By calculating the rational functions associated to a 
$\bZ$-basis $\{x_1,\ldots,x_{d-2},y_1,y_2\}$ for $N$, we have 
the relations
$$D_{x_1}-a_1D_{y_3}+E_1=0,\ \ldots,\ 
D_{x_{d-2}}-a_{d-2}D_{y_3}+E_{d-2}=0,$$
$$D_{y_1}-D_{y_3}+E_{d-1}=0,\ D_{y_1}-D_{y_3}+E_d=0$$
in $\Pic X$, where $E_1,\ldots,E_d$ are torus invariant divisors such that  
$\xSupp E_i\cap S=\emptyset$ for any $1\le i\le d$. 
Therefore, we have 
$$D_{x_1}S=a_1D_{y_3}S,\ \ldots,\ 
D_{x_{d-2}}S=a_{d-2}D_{y_3}S.$$
By these relations, the equality $D_zD_wS=a_za_w$
is obvious.
\end{proof}

\begin{thm}\label{shuteiri2}
Suppose $S\cong F_{\alpha}$, that is, 
a Hirzebruch surface of degree $\alpha$. 
Let $C_{\rm fib}\subset S$ be a fiber of the projection 
$S=F_{\alpha}\to\mathbb{P}^1$, 
while let $C_{\rm neg}$ be the negative section of $S$. Then, 
$I_{S/X}=\alpha(I_{C_{\rm fib}/X})^2
+2I_{C_{\rm fib}/X}I_{C_{\rm neg}/X}.$
\end{thm}

\begin{proof}
Let $\tau=\R_{\ge 0}x_1+\cdots+\R_{\ge 0}x_{d-2}\in\Sigma$ be 
the $(d-2)$-dimensional cone 
associated to $S=S_\tau$, where 
$\tau\cap\G(\Sigma)=\{x_1,\ldots,x_{d-2}\}$. 
Then, there exist exactly four 
maximal cones $\tau+\R_{\ge 0}y_1,\ \tau+\R_{\ge 0}y_2,\ 
\tau+\R_{\ge 0}y_3,\ 
\tau+\R_{\ge 0}y_4\in\Sigma$ which contain $\tau$. 
Put $$y_1+y_3-\alpha y_2+a_1x_1+\cdots+a_{d-2}x_{d-2}=0$$ be 
the wall relation corresponding to $C_{\rm neg}$, 
while  $$y_2+y_4+b_1x_1+\cdots+b_{d-2}x_{d-2}=0$$ be 
the wall relation corresponding to 
$C_{\rm fib}$. 
As in the proof of Theorem \ref{shuteiri}, 
by calculating the rational functions associated to a 
$\bZ$-basis $\{x_1,\ldots,x_{d-2},y_1,y_2\}$ for $N$, 
we have the relations 
$$D_{x_1}S=a_1D_{y_3}S+b_1D_{y_4}S,\ \ldots,\ 
D_{x_{d-2}}S=a_{d-2}D_{y_3}S+b_{d-2}D_{y_4}S,$$
$$D_{y_1}S=D_{y_3}S,\ D_{y_2}=-\alpha D_{y_3}S+D_{y_4}S.$$

First, we remark that 
For any $1\le i,j\le 4$, 
$$D_{y_i}D_{y_j}S=(D_{y_i}|_S)(D_{y_j}|_S)$$
on $S$. So, these intersection numbers can be recovered from 
$I_{S/S}$ (see Example \ref{hanahir}). 

The above relations say that for any $1\le i,j\le d-2$,
$$D_{x_i}D_{x_j}S=\alpha b_ib_j+a_ib_j+a_jb_i,$$
while for any $1\le i\le d-2$,
$$D_{y_1}D_{x_i}=b_i,\ D_{y_2}D_{x_i}=a_i,\ 
D_{y_3}D_{x_i}=b_i,\ D_{y_4}D_{x_i}=a_i+\alpha b_i.$$

On the other hand, 
put $f_1=f_1(X_1,\ldots,X_{d-2}):=
a_1X_1+\cdots+a_{d-2}X_{d-2}$ and 
$f_2=f_2(X_1,\ldots,X_{d-2}):=
b_1X_1+\cdots+b_{d-2}X_{d-2}$. Then, 
$$\alpha(I_{C_{\rm fib}/X})^2
+2I_{C_{\rm fib}/X}I_{C_{\rm neg}/X}$$
$$=
\alpha(Y_2+Y_4+f_1)^2
+2(Y_2+Y_4+f_1)(Y_1+Y_3-\alpha Y_2+f_2)$$
$$=I_{S/S}(Y_1,Y_2,Y_3,Y_4)+\alpha f_2^2+2f_1f_2$$
$$+2Y_1f_2+2Y_2f_1+2Y_3f_2+Y_4(2f_1+2\alpha f_2).$$
This coincides with $I_{S/X}$ by the above calculations.
\end{proof}

\section{2-Fano manifolds}\label{zelda}

As an application of Section \ref{kushi}, we studies on {\em toric $2$-Fano 
manifolds} in this section. The notion of 
$2$-Fano manifolds was introduced in \cite{starr}. 

\begin{defn}
A smooth projecive algebraic variety $X$ is a {\em Fano manifold} if 
its first Chern class ${\rm c}_1(X)=-K_X$ is an ample divisor. 
\end{defn}

\begin{defn}\cite{starr}\label{2fanodef}
A Fano manifold $X$ is a {\em $2$-Fano manifold} if 
its second Chern {\em character} ${\rm ch}_2(X)=\frac{1}{2}({\rm c}_1(X)^2-2{\rm c}_2(X))$ 
is a nef $2$-cocycle. 
\end{defn}

\begin{rem}
Since a $2$-Fano manifod is a Fano manifold by the definition, 
for the classification of toric $2$-Fano manifolds, 
only we have to do is to check the list of toric Fano manifolds. 
The classification of toric Fano manifolds can be done by the 
algorithm of \O bro \cite{obro} for any dimension. 
\end{rem}

For a projective toric manifold $X$, one can easily see that 
${\rm ch}_2(X)=\frac{1}{2}\sum_{i=1}^mD_i^2$, 
where $D_1,\ldots,D_m$ are the torus invariant prime divisors. 
So, the following is immediate.

\begin{prop}
For a torus invariant surface $S\subset X$, put $I_{S/X}:=\sum_{i,j}a_{ij}X_iX_j$. 
Then, $({\rm ch}_2(X)\cdot S)=\frac{1}{2}\sum_{i=1}^ma_{ii}$.
\end{prop}

\begin{say}
First of all, we classify toric $2$-Fano manifolds of Picard number $2$. 
So, let $X$ be a complete toric manifold of $\rho(X)=2$. 
In this case, the structure of $X$ is very simple as follows:

\begin{thm}\cite{klein1}\label{kleinyone}
Every complete toric manifold of Picard number $2$ 
is a projective space bundle over a projective space. 
\end{thm}

By this Theorem \ref{kleinyone}, we can put
$$X=X_\Sigma=\mathbb{P}_{\mathbb{P}^{n-1}}
\left(\mathcal{O}\oplus\mathcal{O}(a_1)\oplus
\cdots\oplus\mathcal{O}(a_{m-1})\right),$$
where $a_1\ge\cdots\ge a_{m-1}\ge 0$, $m+n-2=d:=\dim X$. 
Let 
\begin{eqnarray}
\label{prim01} x_1+\cdots+x_m &=& 0,\\
\label{prim02} y_1+\cdots+y_n &=& 
a_1x_1+\cdots+a_{m-1}x_{m-1},
\end{eqnarray}
be the wall relations of $\Sigma$ 
which correspond to the extremal rays of $\NE(X)$,  
where $${\rm G}(\Sigma)=
\{x_1,\ldots,x_m,y_1,\ldots,y_{n}\}.$$
Let 
$C_1$ and $C_2$ be the extremal 
torus invariant curves corresponding to 
the wall relations $(1)$ and $(2)$, respectively.

First, we determine the extremal rays of $\NE_2(X)$. 
By calculating the rational functions for a $\mathbb{Z}$-basis 
$\{x_1,\ldots,x_{m-1},y_1,\ldots,y_{n-1}\}$, 
we have 
the relations 
$$D_1-D_m+a_1E_n=0,\ldots,
D_{m-1}-D_m+a_{m-1}E_n=0,$$
$$E_1-E_n=0,\ldots,
E_{n-1}-E_n=0$$
in $\N^1(X)$, 
where $D_1,\ldots,D_m,E_1,\ldots,E_n$ are 
torus invariant prime divisors corresponding to 
$x_1,\ldots,x_m,y_1,\ldots,y_n$.
Therefore, for $1\le i,j \le m-1$,
\begin{eqnarray}\label{dou}
D_j=D_i+(a_i-a_j)E_n,
\end{eqnarray}
and
\begin{eqnarray}\label{ikeru}
E_1=E_2=\cdots=E_n.
\end{eqnarray}
On the other hand, every $(d-2)$-dimensional cone $\tau\in\Sigma$ is 
expressed as 
$$\tau=\R_{\ge 0} x_{i_1}+\cdots+\R_{\ge 0} x_{i_k}+
\R_{\ge 0} y_{j_1}+\cdots+\R_{\ge 0} y_{j_l}$$
for some $1\le i_1<\cdots<i_k\le m$, 
$1\le j_1<\cdots<j_l\le n$ 
such that $k<m$, $l<n$ and $k+l=d-2$. 
So, the corresponding torus invariant surface $S_\tau$ 
is expressed as 
$$S_{\tau}=D_{i_1}\cdots D_{i_k}E_{j_1}\cdots E_{j_l}
\in \N_2(X).$$
By using (\ref{dou}) and (\ref{ikeru}), 
any $S_\tau$ is expressed as a linear combination of 
$2$-cycles 
$$D_1\cdots D_pE^q\qquad (p\le m-1,\ q\le n-1,\ p+q=d-2)$$
whose coefficients are non-negative, because  
$i<j$ implies $a_i-a_j\ge 0$. 
Moreover, since $D_1\cdots D_m=E_1\cdots E_n=0$ 
by wall relations (\ref{prim01}) and 
(\ref{prim02}), 
the possibilities for the generators of $\NE_2(X)$ are 
$$
S_1:=D_1\cdots D_{m-3}E^{n-1}, 
S_2:=D_1\cdots D_{m-2}E^{n-2}\mbox{ or}$$
$$
S_3:=D_1\cdots D_{m-1}E^{n-3}.
$$
In fact, the following holds: 
\begin{eqnarray*}
\NE_2(X)  =  & \R_{\ge 0}S_1
+\R_{\ge 0}S_2+
\R_{\ge 0}S_3  & \mbox{if }  m\ge 3,\ n\ge 3. \\
\NE_2(X)  =  & \R_{\ge 0}S_2+
\R_{\ge 0}S_3 & \mbox{if }  m=2,\ n\ge 3. \\
\NE_2(X)  =  & \R_{\ge 0}S_1
+\R_{\ge 0}S_2 & \mbox{if }   m\ge 3,\ n=2.
\end{eqnarray*}
For each case, 
$\dim \N_2(X)=3$, $\dim \N_2(X)=2$ and $\dim \N_2(X)=2$, 
respectively. 
So, 
$\NE_2(X)$ is a {\em simplicial} cone 
for each case, 
and $S_1$, $S_2$ and $S_3$ are extremal surfaces.

Next, 
we will check when $X$ becomes a $2$-Fano manifold.

So, let $C_2$ be the torus invariant curve which generates 
the extremal ray corresponding to 
the wall relation (\ref{prim02}). Then, 
$$
(-K_X\cdot C_2)=n-\left(a_1+\cdots +a_{m-1}\right).
$$
Therefore, 
$X$ is a Fano manifold if and only if 
\begin{eqnarray}\label{mukatuku}
n-\left(a_1+\cdots +a_{m-1}\right)>0.
\end{eqnarray}

Since $S_1\cong S_3\cong \bP^2$, 
$({\rm ch}_2(X)\cdot S_1)\ge 0$ and 
$({\rm ch}_2(X)\cdot S_3)\ge 0$ 
are trivial by Theorem \ref{shuteiri}. 

On the other hand, we can easily check that 
$S_2\cong F_{a_{m-1}}$.  
By Theorem \ref{shuteiri2}, 
we have 
$$I_{S_2}=
a_{m-1}(I_{C_1})^2
+2I_{C_1}I_{C_2}
=
a_{m-1}\left(X_1+\cdots+X_m\right)^2$$
$$+
2\left(X_1+\cdots+X_m\right)\left(Y_1+\cdots+Y_n-
\left(a_1X_1+\cdots+a_{m-1}X_{m-1}\right)\right)
.$$
So, we obtain 
\begin{eqnarray}\label{kusuri}
({\rm ch}_2(X)\cdot S_2)=ma_{m-1}-2(a_1+\cdots+a_{m-1}).
\end{eqnarray}
In (\ref{kusuri}), suppose that $m\ge 3$ and 
$({\rm ch}_2(X)\cdot S_2)\ge 0$. Then,
$$
({\rm ch}_2(X)\cdot S_2)=(m-2)a_{m-1}-2(a_1+\cdots+a_{m-2}).
$$
The assumption $a_1\ge \cdots\ge a_{m-1}\ge 0$ says that 
$a_1=\cdots=a_{m-1}=0,$ that is, 
$X\cong\bP^{m-1}\times\bP^{n-1}$. 
On the other hand, suppose that $m=2$ in (\ref{kusuri}). 
Then, 
$({\rm ch}_2(X)\cdot S_2)=0$, that is, ${\rm ch}_2(X)$ is nef. 

By (\ref{mukatuku}), we can summarize as follows:

\begin{thm}
If $X$ is a toric $2$-Fano manifold of Picard number $2$, 
then $X$ is one of the following$:$
\begin{enumerate}
\item A direct product of projective spaces.
\item $\bP_{\bP^{d-1}}(\mathcal{O}\oplus\mathcal{O}(a))$ 
$(1\le a\le d-1)$.
\end{enumerate}
\end{thm}

\begin{rem}
This calculation shows that 
there exist {\em infinitely} many 
projective toric manifolds of fixed dimension $d$ 
whose second Chern character is nef.  
\end{rem}

\end{say}

\begin{say}
Next, we consider the classification of toric $2$-Fano manifolds of 
a fixed dimension $d$. For $d\le 4$, fortunately, these classifications 
can be done by only Theorems \ref{shuteiri} and \ref{shuteiri2}. 
The classification list is as follows (see \cite{nobili} for the detail):
\begin{table}[h]
\begin{center}
\begin{tabular}{|c||c|c|c|c|}
\hline
$d$ & $1$ & $2$ & $3$ & $4$ \\ \hline\hline
$\#$ of toric Fano & $1$ & $5$ & $18$ & $124$ \\ \hline
$\#$ of toric $2$-Fano & $1$ & $3$ & $8$ & $25$ \\ \hline
\end{tabular}
\end{center}
\end{table}

Since there exist $124$ smooth toric Fano $4$-folds, 
it is hard to check all the smooth toric Fano $4$-folds. 
However, by using the following trivial Lemma \ref{ineq}, 
we can omit a large part of the calculations.

\begin{lem}\label{ineq}
Let $X$ be a $4$-dimensional toric $2$-Fano manifold. 
Then, $${\rm c}_1^4(X)-2{\rm c}_1^2(X){\rm c}_2(X)\ge 0.$$
\end{lem}
For any smooth toric Fano $4$-fold $X$, 
${\rm c}_1^4(X)$ and ${\rm c}_1^2(X){\rm c}_2(X)$ 
are calculated in \cite{bat4}. 
One can see that for $52$ smooth toric Fano $4$-folds, they are 
{\em not} $2$-Fano manifolds by Lemma \ref{ineq}. 

\end{say}

\end{document}